\documentclass[12pt]{article}
\usepackage{latexsym,amssymb,amsmath,amsfonts,pictex,enumerate,amsthm,dsfont, cite}
\usepackage{tikz}
\usetikzlibrary{decorations.markings}
\usepackage{scalerel,stackengine}
\usepackage[active]{srcltx}
\usepackage{mathrsfs}

\theoremstyle{definition}
\newtheorem{theorem}{Theorem}

\newtheorem{lemma}[theorem]{Lemma}

\stackMath
\newcommand\reallywidehat[1]{%
\savestack{\tmpbox}{\stretchto{%
  \scaleto{%
    \scalerel*[\widthof{\ensuremath{#1}}]{\kern-.6pt\bigwedge\kern-.6pt}%
    {\rule[-\textheight/2]{1ex}{\textheight}}
  }{\textheight}%
}{0.5ex}}%
\stackon[1pt]{#1}{\tmpbox}%
}
\begin{document}

\title{H\"older conditions and $\tau$-spikes for analytic Lipschitz functions}
\author{Stephen Deterding \thanks{email: stephen.deterding@westliberty.edu}}

\maketitle

\begin{abstract}
    Let $U$ be an open subset of $\mathbb{C}$ with boundary point $x_0$ and let $A_{\alpha}(U)$ be the space of functions analytic on $U$ that belong to lip$\alpha(U)$, the ``little Lipschitz class''. We consider the condition
    $S= \displaystyle \sum_{n=1}^{\infty}2^{(t+\lambda+1)n}M_*^{1+\alpha}(A_n \setminus U)< \infty,$ where $t$ is a non-negative integer, $0<\lambda<1$, $M_*^{1+\alpha}$ is the lower $1+\alpha$ dimensional Hausdorff content, and $A_n = \{z: 2^{-n-1}<|z-x_0|<2^{-n}\}$. This is similar to a necessary and sufficient condition for bounded point derivations on $A_{\alpha}(U)$ at $x_0$. We show that $S= \infty$ implies that $x_0$ is a $(t+\lambda)$-spike for $A_{\alpha}(U)$ and that if $S<\infty$ and $U$ satisfies a cone condition, then the $t$-th derivatives of functions in $A_{\alpha}(U)$ satisfy a H\"older condition at $x_0$ for a non-tangential approach.
    
\end{abstract}

\section{Introduction}

This paper concerns necessary and sufficient conditions for bounded point derivations on various function spaces. Given a compact subset $X$ in $\mathbb{C}$ and a Banach Space $B$ on $X$, a point $x_0 \in X$ is said to admit a bounded point derivation for $B$ if there exists a constant $C>0$ such that $|f'(x_0)| \leq C ||f||$, for all $f \in B$, where $||\cdot||$ is the norm of the Banach space. Bounded point derivations were originally studied in the case of the space $R(X)$, the uniform closure of rational functions with poles off $X$. An important problem in the theory of rational approximation was to determine conditions for which $R(X) = C(X)$, the space of continuous functions on $X$, because if $R(X) = C(X)$, then every continuous function on $X$ can be uniformly approximated by rational functions with poles off $X$. To solve this problem, the concepts of peak points and non-peak points were developed in papers such as \cite{Bishop, Curtis, Gonchar}. A point $x_0 \in X$ is said to be a peak point for a uniform algebra $A$ on $X$ if there exists $f \in A$ such that $f(x_0) = 1$ and $|f(x)| < 1$ for all other $x \in A$; otherwise it is a non-peak point. Bounded point derivations on $R(X)$ are generalizations of non-peak points, and thus provide information about approximation of derivatives of rational functions. Some of the earlier results in connection with this problem can be found in \cite{Browder} and \cite{Wermer}.  Hence it is of great value to determine necessary and sufficient conditions for bounded point derivations. Moreover, the conditions mentioned in this paper depend only on the point $x_0$ under consideration and the geometry of the set $X$.

\bigskip

Hallstrom was the first to determine necessary and sufficient conditions for bounded point derivations, which he determined for the space $R(X)$, the uniform closure of rational functions with poles outside $X$ \cite{Hallstrom}. These conditions are given in terms of a quantity known as analytic capacity, which is defined as follows. Let $X$ be a compact subset of $\mathbb{C}$. A function $f$ is said to be admissible on $X$ if 

\begin{enumerate}
    \item $|f(z)| \leq 1$ for $z \in \hat{\mathbb{C}} \setminus X$
    \item $f(\infty) = 0$
    \item $f$ is analytic outside $X$
\end{enumerate}

\bigskip \noindent and the analytic capacity of $X$ is denoted by $\gamma(X)$ and defined by 

\begin{equation*}
    \gamma(X) = \sup f'(\infty) = \sup \lim_{z \to \infty} zf(z)
\end{equation*}

\bigskip

\noindent where the supremum is taken over all admissible functions. Hallstrom's conditions for bounded point derivations are summarized in Theorem 1. In it, and throughout the rest of the paper $A_n$ denotes the annulus $\{z: 2^{-n-1} < |z-x_0| < 2^{-n}\}$.

\begin{theorem}

Let $X$ be a compact subset of the complex plane, $x_0 \in X$, and let $t$ be a non-negative integer. Then there exists a bounded point derivation on $R(X)$ at $x_0$ of order $t$ if and only if 

\begin{equation*}
    \sum_{n=1}^{\infty} 2^{(t+1)n} \gamma(A_n \setminus X)< \infty.
\end{equation*}

\end{theorem}

\bigskip

In \cite{O'Farrell1974}, O'Farrell considered the problem of what happens if the integer $t$ in Hallstrom's theorem is replaced with a non-integer. He was able to show two results in opposite directions. The first result concerns sets that satisfy a cone condition. A set $X$ is said to satisfy a cone condition at a point $x_0$ if there exists a cone $\mathscr{C}$ with vertex at $x_0$ and midline $J$ that satisfies the following property: there exists a constant $C>0$ such that if $x\in J$ and $z$ is outside $\mathscr{C}$ then $|x-x_0| \leq C |z-x|$. The midline $J$ is also known as a non-tangential ray to $x_0$ and the limit as $x \to x_0$, $x \in J$ is called a non-tangential limit to $x_0$. O'Farrell's first result shows that replacing the $t$ in Hallstrom's theorem with $t + \lambda$, where $0 < \lambda < 1$, implies a H\"older condition for the $t$-th derivative of $f$, as long as the set $X$ satisfies a cone condition.

\begin{theorem}
\label{O'Farrell1}
Suppose $X$ is a compact subset of $\mathbb{C}$ which satisfies a cone condition at $x_0$ and $J$ is a non-tangential ray to $x_0$. Let $t$ be a non-negative integer and let $0 < \lambda< 1$. If 

\begin{equation}
\label{O'Farrell}
    \sum_{n=1}^{\infty} 2^{(t+\lambda+1)n} \gamma(A_n \setminus X)< \infty
\end{equation}

\bigskip \noindent then there is a constant $C >0$ such that 

\begin{equation*}
    \dfrac{|f^{(t)}(x)-f^{(t)}(x_0)|}{|x-x_0|^{\lambda}} \leq C ||f||_{\infty}
\end{equation*}

\bigskip \noindent for all $x$ in $J$ and $f \in R(X)$.

\end{theorem}
\bigskip

O'Farrell's second result involves representing measures and $\tau$-spikes. Let $\mu$ be a measure and let $ \displaystyle \mu^{\tau}= \int \dfrac{d|\mu(\zeta)|}{|\zeta-z|^{\tau}}$. Then $\mu$ is a representing measure on for a point $x_0 \in X$ on a Banach space $B$ if $\int f d\mu = f(x_0)$ for all $f$ in $B$ and a point $x_0 \in X$ is a $\tau$-spike for $B$ if $\mu^{\tau}(x_0) = \infty$ whenever $\mu$ is a representing measure for $x_0$ on $B$. On $R(X)$ a peak point is a $\tau$-spike for all $\tau>0$. O'Farrell's second result shows that if $x_0$ is a $(t+\lambda)$-spike for $R(X)$ only if \eqref{O'Farrell} holds.

\begin{theorem}
\label{O'Farrell2}
Suppose $X$ is a compact subset of $\mathbb{C}$, $x_0 \in X$,  $0<\lambda<1$ and let $t$ be a non-negative integer. If $\displaystyle \sum_{n=1}^{\infty} 2^{(t+\lambda+1)n} \gamma(A_n \setminus X) = \infty$ then $\mu^{t+\lambda} = \infty$ whenever $\mu$ is a representing measure for $x_0$ on $R(X)$.
\end{theorem}

\bigskip

For the remainder of the paper, we consider bounded point derivations on $A_{\alpha}(U)$, the space of functions that are analytic on an open set $U$ and belong to the ``little Lipschitz class''.  Let $U$ be an open subset in the complex plane and let $0 < \alpha < 1$. A function $f : U \to \mathbb{C}$ satisfies a Lipschitz condition with exponent $\alpha$ on $U$ if there exists $k>0$ such that for all $z,w \in U$
\begin{equation}
\label{lip condition}
|f(z) - f(w)| \leq k |z-w|^{\alpha}.
\end{equation}

\bigskip

Let Lip$\alpha(U)$ denote the space of functions that satisfy a Lipschitz condition with exponent $\alpha$ on $U$. Lip$\alpha(U)$ is a Banach space with norm given by $||f||_{Lip\alpha(U)} = \sup_{U} |f| + k(f)$, where $k(f)$ is the smallest constant that satisfies \eqref{lip condition}. If we let $||f||_{Lip\alpha(U)}' = k(f)$ then $||f||_{Lip\alpha(U)}'$ is a seminorm on Lip$\alpha(U)$. 

\bigskip

 The little Lipschitz class, lip$\alpha(U)$, is the subspace of Lip$\alpha(U)$ which consists of those functions in Lip$\alpha(U)$ that also satisfy the additional property that for each $\epsilon >0$, there exists $\delta >0$ such that for all $z$, $w$ in $U$, $|f(z)-f(w)| \leq \epsilon |z-w|^{\alpha}$ whenever $|z-w| < \delta$. Lipschitz functions form an important class of functions and much work has been done on approximations of Lipschitz functions by rational functions in papers such as \cite{O'Farrell 1975, O'Farrell 1977a, O'Farrell 1977b}.

 \bigskip
 Let $A_{\alpha}(U)$ denote the space of functions that are analytic on $U$ and belong to lip$\alpha(U)$. A point $x_0\in \overline{U}$ is said to admit a bounded point derivation on $A_{\alpha}(U)$ if there exists a constant $C>0$ such that $|f'(x_0)| \leq C ||f||_{Lip\alpha(\mathbb{C})}$ for all $f \in A_{\alpha}(U)$. In \cite{Lord}, Lord and O'Farrell determined the necessary and sufficient conditions for bounded point derivations on $A_{\alpha}(U)$, which is given in terms of an appropriate Hausdorff content.
 
 The Hausdorff content of a set is defined using measure functions. A measure function is a monotone nondecreasing function $h : [0, \infty) \to [0, \infty)$. For example, $r^{\beta}$ is a measure function for $0 \leq  \beta < \infty$. If $h$ is a measure function then the Hausdorff content $M_h$ associated to $h$ is defined by 

\[ M_h(E) = \inf \sum h(\textnormal{diam } B), \]

\bigskip

\noindent where the infimum is taken over all countable coverings of $E$ by balls and the sum is taken over all the balls in the covering. If $h(r) = r^{\beta}$ then we denote $M_h$ by $M^{\beta}$. The lower $(1+ \alpha)$-dimensional Hausdorff content is denoted by $M^{1+\alpha}_*(E)$ and defined by 

\[ M^{1+\alpha}_*(E) = \sup M_h(E), \]

\bigskip

\noindent where the supremum is taken over all measurable functions $h$ such that $h(r) \leq r^{1+\alpha}$ and $r^{-1-\alpha}h(r)$ converges to $0$ as $r$ tends to $0$. The lower $(1+ \alpha)$-dimensional Hausdorff content is a monotone set function; i.e. if $E \subseteq F$ then $M^{1+\alpha}_*(E) \leq M^{1+\alpha}_*(F)$. The result of Lord and O'Farrell characterizing bounded point derivations on $A_{\alpha}(U)$ in terms of Hausdorff content is summarized in the following theorem.

\begin{theorem}
Let $U$ be an open subset of $\mathbb{C}$, $x_0 \in \partial U$ and let $t$ be a non-negative integer. Then there exists a bounded point derivation of order $t$ on $A_{\alpha}(U)$ if and only if 

\begin{equation}
\label{Lord}
    \sum_{n=1}^{\infty} 2^{(t+1)n} M_*^{1+\alpha}(A_n \setminus U) < \infty.
\end{equation}
\end{theorem}

\bigskip

A natural question to ask is what is the significance of a non-integer $t$ in \eqref{Lord}. In analogy with $R(X)$, results similar to Theorem \ref{O'Farrell1} and Theorem \ref{O'Farrell2} should hold for $A_{\alpha}(U)$. In this paper we prove the following results.

\begin{theorem}
\label{main1}
Suppose $U$ is an open subset of $\mathbb{C}$ which satisfies a cone condition at $x_0$ and $J$ is a non-tangential ray to $x_0$. Let $t$ be a non-negative integer and let $0 < \lambda< 1$. If

\begin{equation*}
\sum_{n=1}^{\infty} 2^{(t+1+\lambda)n} M_*^{1+\alpha}(A_n \setminus U) < \infty
\end{equation*}

\bigskip

\noindent then there is a constant $C >0$ such that 

\begin{equation*}
    \dfrac{|f^{(t)}(x)-f^{(t)}(x_0)|}{|x-x_0|^{\lambda}} \leq C ||f||_{Lip\alpha(\mathbb{C})}
\end{equation*}

\bigskip

\noindent for all $x \in J$ and $f \in A_{\alpha}(U)$.

\end{theorem}

\bigskip

\begin{theorem}
\label{main2}
Suppose $U$ is an open subset of $\mathbb{C}$, $x_0 \in \Bar{U}$, $0<\lambda<1$, and let $t$ be a non-negative integer. Also let $\mu^{\beta}(z)= \int \dfrac{d|\mu(\zeta)|}{|\zeta-z|^{\beta}}$. If $\displaystyle \sum_{n=1}^{\infty} 2^{(t+\lambda+1)n} M_*^{1+\alpha}(A_n \setminus U) = \infty$ then $\mu^{t+\lambda} = \infty$ whenever $\mu$ is a representing measure for $x_0$ on $A_{\alpha}(U)$.
\end{theorem}

\bigskip

The next section contains some preliminary lemmas that are used in the proofs of Theorems \ref{main1} and \ref{main2}. In Section 3 we prove Theorem \ref{main1} and in Section 4 we prove Theorem \ref{main2}.

\section{Preliminary lemmas}

Throughout the remainder of the paper, we make use of the following factorization lemma.

\begin{lemma}
\label{factor}
For complex numbers $a$ and $b$ and positive integer $n$,

\begin{equation*}
    a^n - b^n = (a-b)(a^{n-1}+a^{n-2}b + \cdots + ab^{n-2}+  b^{n-1}).
\end{equation*}

\end{lemma}

\bigskip

We will also make use of the following closely related lemma which is less well known.

\begin{lemma}
\label{factor2}
Let $a$ and $b$ be non-zero complex numbers and let $n$ be a negative integer. Then

\begin{equation*}
    a^n-b^n = \dfrac{b-a}{ab} \cdot (a^{n+1}+ a^{n+2}b^{-1} + \ldots + b^{n+1}).
\end{equation*}
\end{lemma}

\begin{proof}
Let $z=a^{-1}$ and $w = b^{-1}$. Then it follows from Lemma \ref{factor} that

\begin{align*}
    a^n-b^n &= z^{-n}-w^{-n}\\
    &=(z-w)(z^{-n-1}+ z^{-n-2}w + \ldots + w^{-n-1})\\
    &=\dfrac{b-a}{ab} \cdot (a^{n+1}+ a^{n+2}b^{-1} + \ldots + b^{n+1}).
\end{align*}

\end{proof}

\bigskip

Another key lemma is the following Cauchy type theorem for Lipschitz functions which appears in the paper of Lord and O'Farrell \cite[pg.110]{Lord}.

\begin{lemma}
\label{Cauchy2}

Let $\Gamma$ be a piecewise analytic curve bounding a region $\Omega \in \mathbb{C}$, and suppose that $\Gamma$ is free of outward pointing cusps. Let $0 < \alpha < 1$ and suppose that $f \in$ lip$\alpha(\mathbb{C})$ is analytic outside a closed region $S$. Then there exists a constant $\kappa >0$ such that

\begin{equation*}
\left| \int f(z)dz\right| \leq \kappa \cdot M_{*}^{1+\alpha}(\Omega \cap S) \cdot ||f||'_{Lip\alpha(\Omega)}.
\end{equation*}

\bigskip

\noindent The constant $\kappa$ only depends on $\alpha$ and the equivalence class of $\Gamma$ under the action of the conformal group of $\mathbb{C}$. In particular this means that $\kappa$ is the same for any curve obtained from $\Gamma$ by rotation or scaling.

\end{lemma}

\bigskip 

The next lemma is an immediate corollary of the cone condition that is applicable in a wide variety of situations.

\begin{lemma}
\label{cone}
Suppose that $J$ is a non-tangential ray to $x_0$ in a cone $\mathscr{C}$, $x \in J$ and $z$ is outside $\mathscr{C}$, then for all positive integers $t$, there exists a constant $C >0$ depending only on $t$ such that 

\begin{equation*}
    \dfrac{1}{|z-x|^t} \leq \dfrac{C}{|z-x_0|^{t}}.
\end{equation*}
\end{lemma}

\begin{proof}
Since $x$ lies on $J$, which is a non-tangential ray to $x_0$, there exists a constant $K>0$ such that for $z \notin \mathscr{C}$, $\dfrac{|x-x_0|}{|z-x|} \leq K$. Thus for $z \notin U$, $\dfrac{|z-x_0|}{|z-x|} \leq 1 + \dfrac{|x-x_0|}{|z-x|} \leq 1 + K$. Hence $\dfrac{1}{ |z-x|^{t}} \leq \dfrac{(1+ K)^t}{|z-x_0|^{t}}$.
\end{proof}

\bigskip

Finally, we will need the following decay lemma which was first proved by Lord and O'Farrell \cite[pg.109]{Lord}.

\begin{lemma}
\label{decay}
Let $\alpha$ be such that $0 < \alpha < 1$, let $K$ be a compact subset of $\mathbb{C}$ and let $f \in $ Lip$\alpha(\mathbb{C})$ be analytic outside $K$ 
and vanish at $\infty$. Then there is a constant $C$ depending on $\alpha$ but not on $K$ or $f$ such that the following estimates hold.

\begin{enumerate}
    \item $||f||_{\infty} \leq C ||f||'_{Lip\alpha(\mathbb{C})} \cdot M_*^{1+\alpha}(K)^{\frac{\alpha}{1+\alpha}}$
    \bigskip
    \item For $z \notin K$, $|f(z)| \leq \dfrac{C ||f||'_{Lip\alpha(\mathbb{C})} \cdot M_*^{1+\alpha}(K)}{\text{dist}(z,K)}$
\end{enumerate}

\end{lemma}

\section{A H\"older condition for derivatives}

In this section we present the proof of Theorem \ref{main1}.

\begin{proof}

To prove Theorem \ref{main1}, we first note that by translation invariance we may suppose that $x_0 = 0$. Moreover by replacing $f$ by $f-f(0)$ if needed, we may suppose that $f(0) = 0$. In addition, we may suppose that $U$ is contained in the ball $\{z:|z| < \frac{1}{2}\}$.

\bigskip

\begin{figure}
\begin{center}
\begin{tikzpicture}

\tikzset{->-/.style={decoration={
  markings,
  mark=at position .5 with {\arrow[scale=2]{>}}},postaction={decorate}}}

\draw [->-,thick,domain=-150:150] plot ({cos(\x)}, {sin(\x)});
   \draw [->-, thick,domain=150:210] plot ({4* cos(\x)}, { 4*sin(\x)});
   
   \draw [->-, thick](-0.866025, 0.5) to (-3.464102, 2);
   \draw [->-, thick] (-3.464102, -2) to (-0.866025, -0.5);

   \filldraw[fill=black, draw=black] (0,0) circle (.08 cm);
   \node [below] (x) at (0, 0){$0$};

   \node [left] (x) at (-4,0){C};
   \node [right] (x) at (1,0){$B_N$};
   \filldraw[fill=black, draw=black] (-2,0) circle (.08 cm);
   \node [below] (x) at (-2, 0){$x$};

\end{tikzpicture}
\begin{tikzpicture}

\tikzset{->-/.style={decoration={
  markings,
  mark=at position .5 with {\arrow[scale=2]{>}}},postaction={decorate}}}

\draw [thick,domain=-150:150] plot ({cos(\x)}, {sin(\x)});
   \draw [  thick,domain=150:210] plot ({4* cos(\x)}, { 4*sin(\x)});
    \draw [ ->-, thick,domain=-150:150] plot ({2* cos(\x)}, { 2*sin(\x)});
     \draw [ ->-, thick,domain=-150:150] plot ({2.5* cos(-\x)}, { 2.5*sin(-\x)});
     
     \draw [->-, thick](-1.732051, 1) to (-2.1650635, 1.25);
     
     \draw [->-, thick] (-2.1650635, -1.25) to (-1.732051, -1);
   \draw (-0.866025, 0.5) to (-3.464102, 2);
   \draw  (-3.464102, -2) to (-0.866025, -0.5);
   \filldraw[fill=black, draw=black] (0,0) circle (.08 cm);
   \node [below] (x) at (0, 0){$0$};
   \node [below] (x) at (0, -1.9){$D_n$};

   \node [left] (x) at (-4,0){C};
   \node [right] (x) at (1,0){$B_N$};
   \filldraw[fill=black, draw=black] (-2,0) circle (.08 cm);
   \node [below] (x) at (-2, 0){$x$};

\end{tikzpicture}

\caption{The contour of integration}
\label{fig3}
\end{center}
\end{figure}
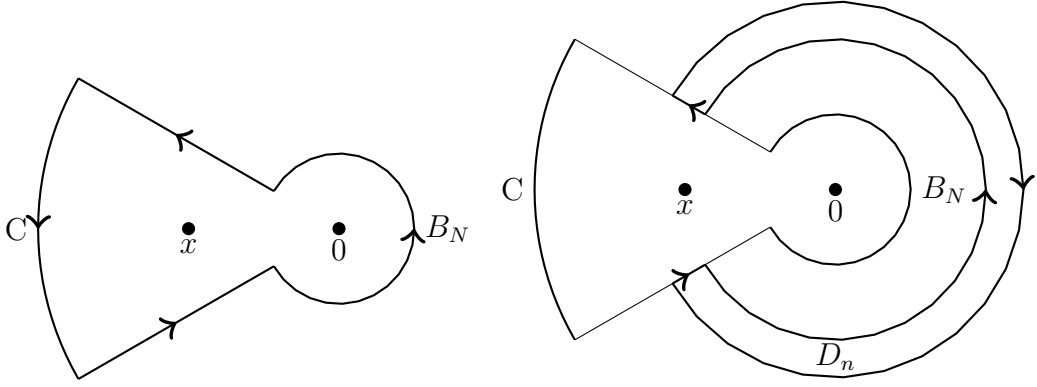

\bigskip

  It is a result of Lord and O'Farrell \cite[Lemma 1.1]{Lord} that $A_{\alpha}(U \cup \{x_0\})$ is dense in $A_{\alpha}(U)$. (Note that in this paper what we refer to as $A_{\alpha}(U)$ is denoted by $a(U)$.) Hence we may suppose that $f$ is analytic at $0$ and thus there is a neighborhood $\Omega$ of $0$ such that $f$ is analytic on $\Omega$. We can further suppose that $U \subseteq \Omega$. Let $B_n$ denote the ball centered at $0$ with radius $2^{-n}$. Then there exists an integer $N>0$ such that $\Omega$ contains $B_N$ and hence $f$ is analytic inside the ball $B_N$. Since $J$ is a non-tangential ray to $0$, it follows that there is a sector in $U$ with vertex at $0$ that contains $J$. Let $C$ denote this sector. It follows from the Cauchy integral formula that
  
  \begin{equation*}
      f^{(t)}(x)-f^{(t)}(0) = \frac{t!}{2 \pi i} \int_{\partial(C \bigcup B_N)} \dfrac{f(z)}{(z-x)^{t+1}} - \dfrac{f(z)}{z^{t+1}}dz
  \end{equation*}
  
  \bigskip \noindent where the boundary is oriented so that the interior of $C \bigcup B_N$ lies always to the left of the path of integration. (See Figure \ref{fig3}.) We can factor an $x$ out of the integrand.

\begin{align*}
f^{(t)}(x)-f^{(t)}(0) &= \frac{t!}{2 \pi i} \int_{\partial(C \bigcup B_N)} \dfrac{f(z)}{(z-x)^{t+1}} - \dfrac{f(z)}{z^{t+1}}dz\\
&= \frac{t!}{2 \pi i} \int_{\partial(C \bigcup B_N)} \dfrac{f(z) \cdot (z^{t+1}-(z-x)^{t+1})}{z^{t+1}(z-x)^{t+1}} dz\\
&= \frac{t!}{2 \pi i} \int_{\partial(C \bigcup B_N)} \dfrac{f(z) \cdot (z-(z-x)) \cdot (z^{t} + z^{t-1}(z-x) + \ldots + (z-x)^{t}) }{z^{t+1}(z-x)^{t+1}} dz\\
&= \frac{t!}{2 \pi i} \int_{\partial(C \bigcup B_N)} f(z) \cdot x \sum_{k=0}^{t}  z^{k-t-1} (z-x)^{-k-1} dz.\\
\end{align*}

\bigskip

\noindent Thus

\begin{equation*}
    \dfrac{f^{(t)}(x)-f^{(t)}(0)}{x^{\lambda}} = \frac{t! x^{1-\lambda}}{2 \pi i} \int_{\partial(C \bigcup B_N)} f(z) \sum_{k=0}^{t}  z^{k-t-1} (z-x)^{-k-1} dz.\\
\end{equation*}

\bigskip

\noindent  Let $D_n = A_n \setminus C$. Then

\begin{align*}
\dfrac{f^{(t)}(x)-f^{(t)}(0)}{x^{\lambda}} &= \frac{t! x^{1-\lambda}}{2 \pi i} \sum_{n=1}^{N} \int_{\partial D_n} f(z) \sum_{k=0}^{t}  z^{k-t-1} (z-x)^{-k-1} dz\\
&+ \frac{t! x^{1-\lambda}}{2 \pi i} \int_{|z|=\frac{1}{2}} f(z) \sum_{k=0}^{t}  z^{k-t-1} (z-x)^{-k-1} dz.
\end{align*}

\bigskip

\noindent We first bound the second integral. Since $x$ lies on $J$, which is a non-tangential ray to $x_0$, it follows from Lemma \ref{cone} that $|z-x|^{-k-1} \leq C |z|^{-k-1}$. Thus by the triangle inequality, 

\begin{align}
\label{Mterm}
    \left|\frac{t! x^{1-\lambda}}{2 \pi i} \int_{|z|=\frac{1}{2}} f(z) \sum_{k=0}^{t}  z^{k-t-1} (z-x)^{-k-1} dz \right| &\leq \dfrac{t! |x|^{1-\lambda}}{2 \pi} \int_{|z|=\frac{1}{2}}\sum_{k=0}^{t}|z|^{-t-2}dz \\
    &\leq C \dfrac{t! |x|^{1-\lambda}}{2 \pi} (t+1) 2^{t+2} \sup_{U} f \leq C \sup_{U} f. \nonumber 
\end{align}

\bigskip

\noindent To bound the sum, we note that since $ \displaystyle x^{1-\lambda} f(z) \sum_{k=0}^{t}  z^{k-t-1} (z-x)^{-k-1}$ is analytic on $D_n \setminus U$ for $M \leq n \leq N$, an application of Lemma \ref{Cauchy2} shows that 

\begin{align}
\label{intbound2}
\left| x^{1-\lambda} \int_{\partial D_n} f(z) \sum_{k=0}^{t}  z^{k-t-1} (z-x)^{-k-1} dz \right| \leq \kappa M^{1+\alpha}(D_n \setminus U) \cdot \left \Vert x^{1-\lambda} f(z) \sum_{k=0}^{t}  z^{k-t-1} (z-x)^{-k-1}  \right \Vert _{Lip\alpha(D_n)}'.
\end{align}

\bigskip

\noindent Recall that the constant $\kappa$ is the same for curves in the same equivalence class. Since the regions $D_n$ differ from each other by a scaling it follows that $\kappa$ doesn't depend on $n$ in \eqref{intbound2}. The remainder of the proof is to show that $\displaystyle \left \Vert x^{1-\lambda} f(z) \sum_{k=0}^{t}  z^{k-t-1} (z-x)^{-k-1} \right \Vert _{Lip\alpha(D_n)}'$ can be bounded by a constant independent of $f$ and $x$. It follows from the definition of the Lipschitz seminorm that

\begin{align}
\label{lipnorm}
  &\left  \Vert x^{1-\lambda}  f(z) \sum_{k=0}^{t}  z^{k-t-1} (z-x)^{-k-1}  \right \Vert _{Lip\alpha(D_n)}'\\
  &= \sup_{z, w \in D_n; z \neq w} \dfrac{|x^{1-\lambda} f(z) \sum_{k=0}^{t}  z^{k-t-1} (z-x)^{-k-1} - x^{1-\lambda} f(w) \sum_{k=0}^{t}  w^{k-t-1} (w-x)^{-k-1}| }{|z-w|^{\alpha}} \nonumber 
\end{align}

\bigskip

\noindent and it follows from an application of the triangle inequality that \eqref{lipnorm} is bounded by

\begin{align}
    & \sup_{z, w \in D_n; z \neq w} \dfrac{| x^{1-\lambda} (f(z)-f(w)) \sum_{k=0}^{t}  z^{k-t-1} (z-x)^{-k-1}|}{|z-w|^{\alpha}}+ \label{long1}\\
    & \sup_{z, w \in D_n; z \neq w}  \dfrac{|x^{1-\lambda}|\cdot|f(w)| \cdot| \sum_{k=0}^{t}  z^{k-t-1} (z-x)^{-k-1}-   w^{k-t-1} (w-x)^{-k-1}| }{|z-w|^{\alpha}}\label{long2}.
\end{align}

\bigskip

\noindent We can determine upper bounds for both \eqref{long1} and \eqref{long2}. To bound \eqref{long1} we recall that 

\begin{equation*}
    ||f||'_{Lip_{\alpha}(D_n)}= \sup_{z, w \in D_n; z \neq w} \dfrac{|f(z)-f(w)|}{|z-w|^{\alpha}}.
\end{equation*}

\bigskip

\noindent It follows from Lemma \ref{cone} that since $x$ is on a non-tangential ray and $z \in D_n$ and $|z-x|^{-k-\lambda} \leq C |z|^{-k-\lambda}$, and it follows from the cone condition that $|x^{1-\lambda}| \leq C |z-x|^{1-\lambda}$. Hence 

\begin{align}
\label{1bound}
\dfrac{| x^{1-\lambda} (f(z)-f(w)) \sum_{k=0}^{t}  z^{k-t-1} (z-x)^{-k-1}|}{|z-w|^{\alpha}} &\leq C \dfrac{|f(z)-f(w)|\cdot |z-x|^{1-\lambda} \cdot |\sum_{k=0}^{t}  z^{k-t-1} (z-x)^{-k-1}|}{|z-w|^{\alpha}} \nonumber\\
&\leq C ||f||'_{Lip_{\alpha}(D_n)} |z|^{-t-1-\lambda} \leq C 2^{n(t+1+ \lambda)} ||f||'_{Lip_{\alpha}(\mathbb{C})}.
\end{align}

\bigskip

Now we bound \eqref{long2}. By the triangle inequality, this is bounded by 

\begin{align}
    &\sup_{z, w \in D_n; z \neq w}  \dfrac{|x^{1-\lambda} f(w)| \cdot| \sum_{k=0}^{t}  (z^{k-t-1} -   w^{k-t-1}) (z-x)^{-k-1}| }{|z-w|^{\alpha}} \label{short1} \\ &+   \dfrac{|x^{1-\lambda} f(w)| \cdot| \sum_{k=0}^{t}  w^{k-t-1} ((z-x)^{-k-1}-  (w-x)^{-k-1})| }{|z-w|^{\alpha}} \label{short2}.
\end{align}

\bigskip

 We first obtain a bound for \eqref{short1}. Since $z,w \in D_n$, $|z|\leq 2 |w|$ and $|w| \leq 2 |z|$. Also it follows from Lemma \ref{cone} that for all integers $m<0$, $|z-x|^{m} \leq C |z|^{m}$ and it follows from the cone condition that $|x^{1-\lambda}| \leq C |z-x|^{1-\lambda}$. Hence it follows from Lemma \ref{factor2} that

\begin{align*}
   &\sup_{z, w \in D_n; z \neq w}  \dfrac{|x^{1-\lambda}f(w)| \cdot| \sum_{k=0}^{t}  (z^{k-t-1} -   w^{k-t-1}) (z-x)^{-k-1}| }{|z-w|^{\alpha}} \\
   &\leq \sup_{z, w \in D_n; z \neq w} \dfrac{|x^{1-\lambda}f(w)| \cdot| \sum_{k=0}^{t}  (w-z)(z^{k-t}+z^{k-t+1}w^{-1} + \ldots + w^{k-t}) (z-x)^{-k-1}| }{|z|\cdot|w| \cdot |z-w|^{\alpha}} \\
    &\leq \sup_{z, w \in D_n; z \neq w} |f(w)| \cdot|z-w|^{1-\alpha} \sum_{k=0}^{t}  (|z|^{k-t-1}|w|^{-1}+|z|^{k-t}|w|^{-2} + \ldots + |z|^{-1}|w|^{k-t-1}) |z-x|^{-k-\lambda}  \\
    &\leq \sup_{z, w \in D_n; z \neq w} C |f(w)| \cdot|z-w|^{1-\alpha} |w|^{-t-2-\lambda}. 
\end{align*}

\bigskip

\noindent Since $f(0) =0$, it follows that for $w \in \mathbb{C}$, $ \dfrac{|f(w)|}{|w|^{\alpha}} \leq ||f||'_{Lip\alpha(\mathbb{C})}$. Hence

\begin{align*}
  \sup_{z, w \in D_n; z \neq w} |f(w)| \cdot|z-w|^{1-\alpha} |w|^{-t-2-\lambda} \leq \sup_{z, w \in D_n; z \neq w} ||f||'_{Lip\alpha(\mathbb{C})} \cdot|z-w|^{1-\alpha} |w|^{-t-2-\lambda+\alpha}. \\
\end{align*}

\bigskip

\noindent Since $z$ and $w$ both belong to $D_n$, $|z-w| \leq C |w|$ and hence  

\begin{align*}
    \sup_{z, w \in D_n; z \neq w} ||f||'_{Lip\alpha(\mathbb{C})} \cdot|z-w|^{1-\alpha} |w|^{-t-2-\lambda+\alpha} \leq  C 2^{n(t+1+\lambda)} ||f||'_{Lip\alpha(\mathbb{C})}.
\end{align*}

\bigskip

\noindent Thus 

\begin{equation}
\label{short1bound}
    \sup_{z, w \in D_n; z \neq w}  \dfrac{|x^{1-\lambda}f(w)| \cdot| \sum_{k=0}^{t}  (z^{k-t-1} -   w^{k-t-1}) (z-x)^{-k-1}| }{|z-w|^{\alpha}} \leq C 2^{n(t+1+\lambda)}||f||'_{Lip\alpha(\mathbb{C})}.
\end{equation}

\bigskip

 We next obtain a bound for \eqref{short2}. Since $z,w \in D_n$, $|z|\leq 2 |w|$, and $|w|\leq 2|z|$. Also it follows from Lemma \ref{cone} that for all integers $m <0$,  $|z-x|^{m} \leq C |z|^{m}$ and $|w-x|^{m} \leq C |w|^{m}$ and it follows from the cone condition that $|x^{1-\lambda}|\leq C |z-x|^{1-\lambda}$. Hence it follows from Lemma \ref{factor2} that

\begin{align*}
   &\sup_{z, w \in D_n; z \neq w}  \dfrac{|x^{1-\lambda} f(w)| \cdot| \sum_{k=0}^{t}  w^{k-t-1} ((z-x)^{-k-1}-  (w-x)^{-k-1})| }{|z-w|^{\alpha}} \\
   &\leq  \dfrac{|f(w)|\cdot |z-x|^{1-\lambda} \cdot| \sum_{k=0}^{t} w^{k-t-1} (z-w)\cdot[(z-x)^{-k}+(z-x)^{-k+1}(w-x)^{-1} + \ldots + (w-x)^{-k}] | }{|z-x|\cdot |w-x|\cdot |z-w|^{\alpha}} \\
    &\leq  \dfrac{|f(w)|}{|z-w|^{\alpha-1}}  \sum_{k=0}^{t} |w|^{k-t-1} (|z-x|^{-k-\lambda}|w-x|^{-1}+|z-x|^{-k+1-\lambda}|w-x|^{-2} + \ldots + |z-x|^{-\lambda}|w-x|^{-k-1})   \\
    &\leq \sup_{z, w \in D_n; z \neq w} C |f(w)| \cdot|z-w|^{1-\alpha} |w|^{-t-2-\lambda} 
\end{align*}

\bigskip

\noindent Since $f(0) =0$ it follows that for $w \in \mathbb{C}$, $ \dfrac{|f(w)|}{|w|^{\alpha}} \leq ||f||'_{Lip\alpha(\mathbb{C})}$. Hence

\begin{align*}
  \sup_{z, w \in D_n; z \neq w} |f(w)| \cdot|z-w|^{1-\alpha} |w|^{-t-2-\lambda} \leq \sup_{z, w \in D_n; z \neq w} ||f||'_{Lip\alpha(\mathbb{C})} \cdot|z-w|^{1-\alpha} |w|^{-t-2-\lambda+\alpha}. \\
\end{align*}

\bigskip

\noindent Since $z$ and $w$ both belong to $D_n$, $|z-w| \leq C |w|$ and hence  

\begin{align*}
   \sup_{z, w \in D_n; z \neq w} ||f||'_{Lip\alpha(\mathbb{C})} \cdot|z-w|^{1-\alpha} |w|^{-t-2-\lambda+\alpha} \leq C 2^{n(t+1+\lambda)}||f||'_{Lip\alpha(\mathbb{C})}.
\end{align*}

\bigskip

\noindent Thus

\begin{equation}
    \label{short2bound}
    \sup_{z, w \in D_n; z \neq w}  \dfrac{|x^{1-\lambda} f(w)| \cdot| \sum_{k=0}^{t}  w^{k-t-1} ((z-x)^{-k-1}-  (w-x)^{-k-1})| }{|z-w|^{\alpha}} \leq 2^{n(t+1+\lambda)}||f||'_{Lip\alpha(\mathbb{C})}.
\end{equation}

\bigskip

\noindent By applying \eqref{short1bound} and \eqref{short2bound} it follows that  \eqref{long2} is bounded by $C 2^{n(t+1+\lambda)}||f||'_{Lip\alpha(\mathbb{C})}$ and it follows from this and \eqref{1bound} that  

\begin{equation*}
    \left|x^{1-\lambda} \int_{\partial D_n} f(z) \sum_{k=0}^{t}  z^{k-t-1} (z-x)^{-k-1} dz \right| \leq C 2^{n(t+1+\lambda)}M_*^{1+\alpha}(D_n \setminus U)||f||'_{Lip\alpha(\mathbb{C})}.
\end{equation*}

\bigskip

\noindent Since Hausdorff content is monotone, $M_*^{1+\alpha}(D_n \setminus U) \leq M_*^{1+\alpha}(A_n \setminus U)$ and hence by the hypothesis of the theorem,

\begin{align*}
    \left|x^{1-\lambda} \sum_{n=M}^N\int_{\partial D_n} f(z) \sum_{k=0}^{t}  z^{k-t-1} (z-x)^{-k-1} dz \right| &\leq C \sum_{n=1}^{\infty} 2^{n(t+1+\lambda)}M_*^{1+\alpha}(A_n \setminus U)||f||'_{Lip\alpha(\mathbb{C})}\\
    &\leq C||f||'_{Lip\alpha(\mathbb{C})}.
\end{align*}

\noindent and thus it follows from this and \eqref{Mterm} that 

\begin{align*}
    \dfrac{|f^{(t)}(x) -f^{(t)}(0)|}{|x|^{\lambda}} \leq C||f||_{Lip\alpha(\mathbb{C})}
\end{align*}

\bigskip

\noindent for all $f\in A_{\alpha}(U)$ and $x \in J$.

\end{proof}
\section{Representing measures and bounded point derivations}

We now prove Theorem \ref{main2}.

\begin{proof}
As before we will assume that $x_0 = 0$ and $X$ is contained in the ball $\{z: |z| < \frac{1}{2}\}$. Choose a sequence $\epsilon_n \to 0$ such that $\displaystyle \sum_{n=1}^{\infty} 2^{(t+\lambda+1)n} \epsilon_n M_*^{1+\alpha}(A_n \setminus U) = \infty$ and $2^{(t+\lambda+1)n} \epsilon_n M_*^{1+\alpha}(A_n \setminus U) \leq 1$ for all $n$. Then for each integer $N$, there exists an integer $M >N$ such that

\begin{equation*}
    1 \leq \sum_{n=N}^M 2^{(t+\lambda+1)n}\epsilon_n M_*^{1+\alpha}(A_n \setminus U) \leq 2.
\end{equation*}

\bigskip

\noindent By Frostman's Lemma \cite[\textbf{2.2}]{Lord}, for each integer $n$ with $N \leq n \leq M$ there exists a positive measure $\nu_n$ with support on $A_n \setminus U$ such that $\int  \nu_n = C \epsilon_n M_*^{1+\alpha}(A_n \setminus U)$, where the constant $C>0$ and does not depend on $n$ or $U$. Now define a function $f_n$ by 

\begin{equation*}
    f_n(z) = \int  \dfrac{d\nu_n(\zeta)}{z-\zeta}.
\end{equation*}

\bigskip

\noindent $f_n$ belongs to $A_{\alpha}(U)$ and is analytic off $A_n$. In addition, it follows from Lemma \ref{decay} that $|f_n(z)| \leq C 2^{-\lambda n}$ and if $z\notin A_{n-1} \cup A_n \cup A_{n+1}$, then $|f_n(z)| \leq \dfrac{C M_*^{1+\alpha}(A_n \setminus U)}{dist(z, A_n)}$.

\bigskip

Now define $g_N(z)$ by 

\begin{equation*}
    g_N(z) = |z|^{\lambda} z^{t+1} \sum_{n=N}^{M} 2^{(t+ \lambda +1)n} f_n(z)
\end{equation*}

\bigskip

\noindent and consider the sequence $\{g_N\}_{1}^{\infty}$. We will show that this sequence is uniformly bounded on the unit disk.

\bigskip

Suppose that $z \in A_j$ for some positive integer $j$. By Lemma \ref{decay}, if $n \neq j-1, j, j+1$ then 

\begin{equation*}
    |f_n(z)| \leq C 2^j M_*^{1+\alpha}(A_n \setminus U)
\end{equation*}

\bigskip

\noindent and if $n = j-1, j, j+1$ then 

\begin{equation*}
    |f_n(z)| \leq C 2^{-\lambda n} \leq C
\end{equation*}

\bigskip

\noindent since $n$ is a positive integer. From these estimates it follows that

\begin{equation*}
    |g_N(z)| \leq |z|^{\lambda + t +1} \left(3C \cdot 2^{(t+\lambda+1)j}+ \sum_{n=N}^{M} C 2^{(t+\lambda+1)n} 2^j M_*^{1+\alpha}(A_n \setminus U)\right).
\end{equation*}

\bigskip \noindent Since $z \in A_j$, $|z|^{\lambda +t+1} \leq C 2^{-(t+\lambda+1)j}$ and

\begin{align*}
    |g_N(z)| & \leq 3C+ |z|^{\lambda + t} 2C.
\end{align*}

\bigskip

\noindent Thus $\{g_N(z)\}$ is uniformly bounded on the unit disk. Next define $h_N(z)$ by $h_N(z) = |z|^{-\lambda}z g_N(z)$, and consider the sequence $\{h_N(z)\}$. Because $g_N(z)$ is uniformly bounded on the unit disk, so also is $h_N(z)$ and since $h_N(z)$ is analytic outside the ball $\{|z| \leq 2^{-N}\}$, a subsequence (also denoted $\{h_N\}$) converges pointwise on $\mathbb{C}\setminus \{0\}$ to a function $h$ which is also analytic on $\mathbb{C}\setminus \{0\}$. Moreover, $g_N$ is uniformly bounded near $0$ for each $N$ and hence $h(0) = 0$. Thus $h$ is entire. 

\bigskip

Let $k_N(z)$ be defined by $k_N(z) = z^{-t-2}h_N(z)$. Then 

\begin{align*}
    k_N'(\infty) &= \lim_{z \to \infty} z k_N(z)\\
    &= \lim_{z \to \infty} z^{-t-1} h_N(z)\\
    &= \lim_{z \to \infty} |z|^{-\lambda} z^{-t} g_N(z)\\
    &= \lim_{z \to \infty} z \sum_{n=N}^M 2^{(t+\lambda+1)n}f_n(z).
\end{align*}

\bigskip

\noindent Since $\displaystyle \lim_{z \to \infty} zf_n(z) = \int \nu_n = C \epsilon_n M_*^{1+\alpha}(A_n \setminus U)$, it follows that $0 < C \leq |k_N'(\infty)| \leq 2C$. Thus passing to a second subsequence, still denoted by $\{k_N'\}$, we find that $\{k_N'(\infty)\} $ converges to $\beta$ for some $C \leq \beta \leq 2C$. Thus $\displaystyle \lim_{z \to \infty} z k_N(z) = \beta$ and hence $g_N(z)$ converges pointwise to $\beta |z|^{\lambda}z^{t}$.

\bigskip

 Assume that there exists a measure $\mu$ which represents $0$ on $A_{\alpha}(U)$ such that $\mu^{t+\lambda} <  \infty$. Then $|z|^{-\lambda}z^{-t}\mu$ is a finite measure. Let $L_N(z) = |z|^{-\lambda}g_N(z)$. Then $L_N(z)$ is analytic in a neighborhood of $0$ and belongs to $A_{\alpha}(U)$. Hence

\begin{align*}
    0 = L_N^{(t)}(0) &= t! \int \dfrac{L_N(z)}{z^t} d\mu(z) \\
    &= t! \int \dfrac{g_N(z)}{|z|^{\lambda}z^t} d\mu(z) \\
    &\to t! \int \beta d\mu(z) = t! \cdot \beta,
\end{align*}

\bigskip

\noindent which is a contradiction. Hence $\mu^{t+\lambda} = \infty$.

\end{proof}

\end{document}